\documentclass{amsart}
\usepackage[utf8x]{inputenc}
\usepackage{amsmath, amsthm, amssymb}
\usepackage[usenames,dvipsnames,svgnames,table]{xcolor}
\usepackage[margin=1.3in]{geometry}
\usepackage{cleveref}

\newtheorem{theorem}{Theorem}[section]
\newtheorem{proposition}[theorem]{Proposition}
\newtheorem{lemma}[theorem]{Lemma}
\newtheorem{definition}[theorem]{Definition}

\newtheorem{corollary}[theorem]{Corollary}

\newcommand{\R}{\mathbb R}

\newcommand{\vv}{\langle v\rangle}

\newcommand{\dd}{\, \mathrm{d}}

\newcommand{\tr}{\mbox{tr}}

\numberwithin{equation}{section}

\title[A continuation criterion for the Landau equation]{A continuation criterion for the Landau equation with very soft and Coulomb potentials}

\author[Stanley Snelson]{Stanley Snelson}
\address{Department of Mathematics and Systems Engineering, Florida Institute of Technology, Melbourne, FL 32901}
\email{ssnelson@fit.edu}

\author{Caleb Solomon}
\address{Department of Mathematics and Systems Engineering, Florida Institute of Technology, Melbourne, FL 32901}
\email{csolomon2017@my.fit.edu}

\thanks{SS would like to thank Chris Henderson and Will Golding for interesting discussions about this work.}

\thanks{SS was partially supported by NSF grants DMS-2213407 and DMS-2511236.}

\begin{document}

\maketitle

\begin{abstract}
We consider the spatially inhomogeneous Landau equation in the case of very soft and Coulomb potentials, $\gamma \in [-3,-2]$. We show that solutions can be continued as long as the following three quantities remain finite, uniformly in $t$ and $x$: (1) the mass density, (2) the velocity moment of order $s$ for any small $s>0$, and (3) the $L^p_v$ norm for any $p>3/(5+\gamma)$. In particular, we do not require a bound on the energy density. 
\end{abstract}

\section{Introduction}

We consider the Landau equation, a collisional kinetic model from plasma physics. The unknown function $f(t,x,v)\geq 0$ models the density of particles at time $t\geq 0$, location $x\in \R^3$, and velocity $v\in \R^3$. The equation reads
\begin{equation}\label{e:landau}
\partial_t f + v\cdot\nabla_x f = Q(f,f),
\end{equation}
where $Q$ is the bilinear Landau collision operator, defined for functions $f,g:\R^3\to\R$ by
\begin{equation}\label{e:collision}
Q(f,g) =  \nabla_v \cdot \left(\int_{\R^3} a(v-w) [f(w) \nabla_v g(v) - f(v)\nabla_w g(w)] \dd w\right),
 \end{equation}
Here, the matrix $a$ is defined by
\[
a(z) = a_\gamma |z|^{\gamma+2} \left( I - \frac{z\otimes z}{|z|^2}\right), \quad z\in \R^3,
\]
for some $\gamma \in [-3,1]$, and $a_\gamma >0$ is a constant depending on $\gamma$. 

This article is concerned with the case $\gamma \in [-3, -2]$, which is known as {\it very soft potentials}. This case is the most difficult to analyze mathematically, because the singularity of order $\gamma+2$ in $a(z)$ is the most severe. Included in our analysis is the case $\gamma = -3$ ({\it Coulomb potentials}), which is the most physically relevant case as a model of plasmas.

We are concerned with the {\it large-data} regime, where $f$ and the initial data are not assumed to be close to an equilibrium state. The equilibrium states for \eqref{e:landau} are known as {\it Maxwellians} and take the form  $c_1 e^{-c_2 |v-v_0|^2}$ for $c_1, c_2>0$ and $v_0\in \R^3$. It has been known since the work of Guo \cite{guo2002periodic} in 2002 that a global solution exists if the initial data is sufficiently close to a Maxwellian. (See \cite{carrapatoso2016cauchy, CM2017verysoft, kim2020landau, duan2019mild, guo2020landau, golding2023global} and the references therein  for futher results on the close-to-equilibrium regime, and \cite{luk2019vacuum, chaturvedi2020vacuum} for global solutions close to the vacuum state $f\equiv 0$.)
 
By contrast, global existence of classical solutions in the large-data regime is a difficult unsolved problem. In recent years, there has been partial progress 
in the form of conditional regularity results and continuation criteria, see e.g. \cite{golse2016, cameron2017landau, HST2018landau, HST2019rough}. To discuss these results, let us define for any solution $f$ the densities
 \begin{align*}
 M_f(t,x) = \int_{\R^3} f(t,x,v) \dd v, \qquad &\text{(mass density)}\\
 E_f(t,x) = \int_{\R^3} |v|^2 f(t,x,v) \dd v, \qquad&\text{(energy density)}.
 \end{align*}
 The best continuation criterion currently available seems to be \cite[Theorem 1.3]{HST2019rough}, which says that solutions can be continued for as long as the following quantity remains finite:
 \begin{equation}\label{e:old-condition}
 \begin{cases} 
 \displaystyle\sup_{t\in [0,T],x\in \R^3} [M_f(t,x) + E_f(t,x)], & \text{ if } \gamma \in (-2,0),\\
 \displaystyle\sup_{t\in [0,T],x\in \R^3} [M_f(t,x) + \|f(t,x,\cdot)\|_{L^q_v(\R^3)}], & \text{ if } \gamma \in [-3,-2],
 \end{cases}
 \end{equation}
 where 
 \[
\begin{cases} 
q > \dfrac 3 {3+\gamma}, &\gamma\in (-3,-2],\\
q = \infty, &\gamma = -3.
\end{cases}
\] 
 The goal of this paper is to improve the continuation criterion \eqref{e:old-condition} in the case $\gamma \in [-3,-2]$.


 \subsection{Main results}
 
Let us define the notion of solution that we use throughout the paper:
 
 \begin{definition}\label{d:solution}
 For any $T_1 < T_2$, we say that $f:[T_1,T_2]\times \R^6\to [0,\infty)$ is a classical solution of the Landau equation if $f \in (C^1_{t,x} C^2_v)_{\rm loc}([T_1,T_2]\times \R^6)$, $\int_{\R^3} f(t,x,v) \dd v \in L^\infty_{\rm loc}([T_1,T_2]\times\R^6)$, and $f$ satisfies \eqref{e:landau} pointwise for all $(t,x,v) \in (T_1,T_2)\times \R^6$. 
 \end{definition}
 Note that the condition on $\int_{\R^3} f(t,x,v) \dd v$ in Definition \ref{d:solution} ensures $f$ has enough decay in $v$ so that the integral in \eqref{e:collision} converges. This definition implies in particular that the initial condition $f(0,x,v)$ is in $(C^1_{x} C^2_v)_{\rm loc}(\R^6)$. When we discuss solutions in a smaller set, such as the cylinder $(-1,0]\times B_1 \times \R^3$, it is understood that this is the restriction to the smaller set of a solution as in Definition \ref{d:solution} on some $[T_1,T_2]\times\R^6$.
 
Our first main results are about upper bounds in $L^\infty$ that depend only on the initial data and on relatively weak conditional bounds on $f$. There are two versions of the result: first, we have a bound with a strong decay assumption at time zero, but a weaker conditional assumption on $f$ for positive time:
 
 \begin{theorem}\label{t:upper}
Let $\gamma \in [-3,-2]$, and let $f\geq 0$ be a classical solution to the Landau equation on $[0,T]\times\R^6$, as in Definition \ref{d:solution}, for some $T>0$. Assume that the initial data $f_{\rm in}(x,v) = f(0,x,v)$ satisfies the lower bound
 \begin{equation}\label{e:fin-lower}
 f_{\rm in}(x,v) \geq \ell, \quad x\in \R^3, v\in B_\rho(0),
 \end{equation}
 for some $\ell, \rho>0$, as well as the upper bound
 \begin{equation}\label{e:fin-upper}
 f_{\rm in} \leq C_0 e^{-\mu |v|^2},
	 \end{equation}
 for some $C_0, \mu >0$. Furthermore, assume that $f$ satisfies the upper bounds
 \begin{equation}\label{e:cond}
 M_f(t,x)\leq M_0, \quad \int_{\R^3} |v|^s f(t,x,v)\dd v \leq S_0, \quad \|f(t,x,\cdot)\|_{L^{p+\delta}(\R^3)} \leq P_0,  
 \end{equation}
uniformly in $x\in \R^3$ and $t\in [0,T]$, for some $s \in (0,2)$ and $\delta>0$, where $p= \dfrac 3 {5+\gamma}$.
 
 Then $f$ satisfies a global upper bound 
 \[
 f(t,x,v) \leq C,
 \]
 for some $C>0$ depending on $\gamma$, $\ell$, $\rho$, $C_0$, $\mu$, $s$, $\delta$, $T$, and the constants in \eqref{e:cond}.
 \end{theorem}

 \begin{theorem}\label{t:other-upper}
  Let $\gamma\in [-3,-2]$, and let $f\geq 0$ be a classical solution to the Landau equation on $[0,T]\times\R^6$ as in Definition \ref{d:solution}, for some $T>0$. Assume that the initial data $f_{\rm in}(x,v) = f(0,x,v)$ satisfies the lower bound \eqref{e:fin-lower}, and that $f_{\rm in}$ satisfies the decay $f_{\rm in}(x,v) \leq K \vv^{-q}$ for some $q$ large enough, depending only on $\gamma$.
  
Furthermore, assume that $f$ satisfies the upper bounds
\begin{equation}\label{e:other-cond}
 M_f(t,x)\leq M_0,  \quad \|\langle \cdot \rangle^{-19\gamma/(p+\delta)} f(t,x,\cdot)\|_{L^{p+\delta}(\R^3)} \leq P_0, 
\end{equation}
uniformly in $x\in \R^3$ and $t\in [0,T]$, with $p$ and $\delta$ as in Theorem \ref{t:upper}. 

Then $f$ satisfies a global upper bound
\[
f(t,x,v) \leq C,
\]
for some $C>0$ depending on $\gamma$, $\ell$, $\rho$, $\delta$, $T$, and the constants in \eqref{e:other-cond}.
\end{theorem}

As mentioned earlier, the benefit of Theorem \ref{t:other-upper} is that it requires less decay for the initial condition, at the cost of a stronger assumption of decay for positive times.

On a more technical note, the assumption of polymomial decay at time zero is not used quantitatively in Theorem \ref{t:other-upper}, and is only required in order to apply results from the literature about propagating lower bounds forward in time. By contrast, the Gaussian decay assumption \eqref{e:fin-upper} plays an important quantitative role in Theorem \ref{t:upper}.

When combined with the results of \cite{henderson2017smoothing}, our Theorems \ref{t:upper} and \ref{t:other-upper} imply that $f$ satisfies regularity estimates of all orders on $[t,T]\times\R^6$ for any $t>0$, with constants depending only on $t$, $T$, the initial data, and the constants in \eqref{e:cond}.  Furthermore, the continuation criterion \eqref{e:old-condition} from \cite[Theorem 1.3]{HST2019rough} applies to $f$, because of our assumptions on $f_{\rm in}$. Therefore, bounding the $L^p_v$ norm in \eqref{e:old-condition} with Theorem \ref{t:upper}, we immediately obtain:
 
 \begin{corollary}\label{c:cont}
Let $\gamma \in [-3,-2]$, and let $f$ be a classical solution to the Landau equation as in Definition \ref{d:solution}, with $f_{\rm in}$ satisfying the hypotheses \eqref{e:fin-lower} and \eqref{e:fin-upper} from Theorem \ref{t:upper}. 

If $T_*< \infty$ is the maximal time of existence of the solution $f$, i.e. if $f$ cannot be extended to a solution on $[0,T_*+\tau)\times\R^6$ for any $\tau>0$, then one of the inequalities in \eqref{e:cond} must degenerate as $t\nearrow T_*$, i.e. either
 \[
\sup_{x\in\R^3} M_f(t,x) \nearrow +\infty \,\, \text{ or } \,\, \sup_{x\in \R^3} \int_{\R^3} |v|^s f(t,x,v) \dd v \nearrow +\infty \,\,\text{ or } \,\, \sup_{x\in\R^3} \|f(t,x,\cdot)\|_{L^{p+\delta}_v(\R^3)}\nearrow +\infty,
\]
as $t\nearrow T_*$. 

If we replace the upper bound assumption \eqref{e:fin-upper} with the weaker polynomial decay estimate $f_{\rm in} (x,v) \leq K \vv^{-q}$, with $q$ as in Theorem \ref{t:other-upper}, then the conclusion is modified as follows: if $T_*$ is the maximal time of existence of $f$, then one of the inequalities in \eqref{e:other-cond} must degenerate as $t\nearrow T_*$.

 \end{corollary}

Unlike $q$ in \eqref{e:old-condition}, the critical exponent $p= \dfrac 3 {5+\gamma}$ in Theorem \ref{t:upper}, Theorem \ref{t:other-upper}, and Corollary \ref{c:cont} does not approach $+\infty$ as $\gamma \searrow -3$.
 
We should note that our lower bound condition \eqref{e:fin-lower} on the initial data could be relaxed to allow the presence of vacuum regions, by applying the positivity-spreading result of \cite[Theorem 1.3]{HST2018landau}. For the sake of a simple statement of our results, we focus instead on the continuation of solutions with nice (but large) initial data.

 \subsection{Comparison with homogeneous Landau}
 
 It is interesting to compare these results to what is known for the spatially homogeneous Landau equation, which arises from assuming the solution of \eqref{e:landau} is constant in $x$. Then $f(t,v)$ satisfies
 \begin{equation}\label{e:homogeneous}
 \partial_t f = Q(f,f).
 \end{equation}
Compared to the full inhomogeneous Landau equation, more results about existence and regularity are available for \eqref{e:homogeneous}, see \cite{fournier2010uniqueness, desvillettes2015landau, silvestre2015landau, gualdani2017landau, CDE2017landau, carillo2020gradient, GGIV2022landau, golse2022local, desvillettes2023mono} and the references therein.
In particular, this equation is known to be globally well-posed  for any $\gamma\in [-3,1]$ : see \cite{villani1998landau, desvillettes2000landau, alexandre2015apriori, Wu2014global} and \cite{guillen2023landau} which showed existence for the last remaining case, $\gamma \in [-3,-2)$, by proving that the Fisher information $\int_{\R^3} |\nabla_v f|^2 /f \dd v$ is nonincreasing. The $L^{3/(5+\gamma)}$ norm, which is a borderline in the results of the current article about inhomogeneous Landau, is in some sense a critical norm for the Landau equation, and appears often in the study of homogeneous Landau. Prior to the breakthrough of \cite{guillen2023landau}, the norm $L^\infty_t L^q_v([0,T]\times \R^3)$ was the best known continuation criterion for \eqref{e:homogeneous}  \cite{gualdani2017landau, ABDL, golding2023local}. A related borderline appears in the recent work \cite{ggl-arma}, which (working in the Coulomb case $\gamma = -3$) enlarged the space of allowable initial data for \eqref{e:homogeneous} to a weighted $L^{3/2}$ norm, while recovering global existence using some ideas from \cite{guillen2023landau}. For local-in-time existence, $L^{3/2}_v$ seems to be the weakest possible space with current techniques, even in the space homogeneous case.



 Note that $q>3/(5+\gamma)$ is the minimal condition required so that  $\|f\|_{L^q_v}$ controls the convolution $f\ast |v|^{\gamma+2}$ (see \eqref{e:collision} and \eqref{e:abc}) uniformly from above. This condition also appears as a borderline in the result of \cite{bedrossian2022vpl}, which ruled out some approximately self-similar blowup solutions.

 Recently, Alonso-Bagland-Desvillettes-Lods \cite{ABDL} have derived a Prodi-Serrin-like condition for homogeneous Landau: if a solution $f$ (up to a polynomial weight) lies in $L^r([0,T],L^q(\R^3))$ for some $q>1$ and $r\geq 1$ satisfying
 \[
 \frac 2 r + \frac 3 q = 5 + \gamma,
 \]
 then $f$ is bounded in $L^\infty$ for positive times and can therefore be continued past time $T$. (Some unweighted Prodi-Serrin conditions were subsequently derived in \cite{golding2023local}.) 
 It would be interesting to derive this kind of condition for the inhomogeneous Landau equation \eqref{e:landau}, using mixed $(t,x,v)$ norms of the form $L^r_t L^q_x L^p_v([0,T]\times\R^3\times\R^3)$ for some $r, q, p$. 
 \subsection{Proof ideas}

The philosophy of our proof is to leverage the diffusive properties of the collision term $Q(f,f)$, while exploiting the {\it nonlinear structure} of this diffusion more fully, compared to some previous works on the large-data case of the Landau equation.

To explain what this means, let us write the bilinear collision operator \eqref{e:collision} in the usual way as a diffusion operator, either in divergence form
\[
Q(f,g) = \nabla_v\cdot(\bar a^f\nabla_v g) + \bar b^f\cdot \nabla_v g + \bar c^f g,
\]
or nondivergence form
\[
Q(f,g) = \tr(\bar a^f D_v^2 g) + \bar c^f g,
\]
where 
\begin{equation}\label{e:abc}
\begin{split}
\bar a^f &= a_\gamma\int_{\R^3} |w|^{\gamma+2}\Pi(w) f(v-w) \dd w,\\
\bar b^f &= b_\gamma \int_{\R^3} |w|^\gamma w f(v-w) \dd w,\\
\bar c^f &= 
\begin{cases} 
c_\gamma \int_{\R^3} |w|^\gamma f(v-w) \dd w, & \gamma > -3,\\
f, & \gamma = -3,
\end{cases}
 \end{split}
 \end{equation}
 where $b_\gamma$ and $c_\gamma >0$ are constants depending on $\gamma$, and $\Pi(w) = \left( I - \dfrac{w\otimes w}{|w|^2}\right)$. 
 
 Our argument proceeds in the following steps:
 
 \begin{enumerate}
 
\item 
 First, prove a local $L^\infty$ estimate for $f$ by a Moser iteration argument that exploits the gain in regularity/integrability provided by velocity averaging. This argument is inspired by Golse-Imbert-Mouhot-Vasseur \cite[Theorem 12]{golse2016}, who considered linear kinetic Fokker-Planck equations of the form 
 \begin{equation}\label{e:linear}
 \partial_t f + v\cdot\nabla_x f = \nabla_v\cdot (A\nabla_v f) + B\cdot \nabla_v f + s
 \end{equation}
  for general coefficients $A$, $B$, $s$, and then applied their estimate to the Landau equation by placing suitable conditions on $f$ so that the coefficients $\bar a^f$, $\bar b^f$, and $\bar c^f$ in \eqref{e:abc} are well-behaved. This works well when $\gamma \geq -2$, but when $\gamma< -2$, a bound for $f$ in $L^q_v$ is needed to control the coefficients, with $q$ as in \eqref{e:old-condition}, and we would like to avoid this assumption. We address this problem by ``remembering'' the coupling between $f$ and the coefficients earlier in the proof, which leads to an estimate whose constant has  a less severe  dependence on higher integrability norms of $f$.\footnote{It turns out to be more convenient to allow this constant to depend on $\|f\|_{L^\infty}$ rather than $\|f\|_{L^q_v}$. The next two steps of the argument will remove the dependence on the $L^\infty$ norm.} 

\item 

Next, we improve the local estimate using scaling techniques. It is well known that if $f$ solves the Landau equation \eqref{e:landau}, then for any $r>0$ and $\alpha \in \R$, the function 
\[
f_r(t,x,v) = r^{\alpha + 3 + \gamma} f(r^\alpha t, r^{1+\alpha} x, r v)
\]
is also a solution. By choosing a convenient value of $\alpha$ and a scale $r$ that depends on the $L^\infty$ norm of $f$, we obtain a pointwise upper bound of the form $f(t,x,v) \leq C\|f\|_{L^\infty}^\beta$ for some $\beta < 1$, which would imply an unconditional $L^\infty$ estimate by taking the supremum over $(t,x,v)$.

This scaling argument should be compared to \cite{cameron2017landau}, which applied rescaling techniques to estimates for the {\it linear} equation \eqref{e:linear}. Again, this worked well only when $\gamma \geq -2$. 

\item Unfortunately, the constant $C$ in the previous step blows up for large $|v|$, so we cannot naively take the supremum over $v$. We get around this problem in two different ways. First, by absorbing the growing term into the conditional $L^{p+\delta}_v$ norm as a weight, we obtain Theorem \ref{t:other-upper}. But in order to prove Theorem \ref{t:upper}, which features an unweighted conditional norm of $f$, we need to use pointwise decay of $f$, which is why we need to assume decay for $f_{\rm in}$ in Theorem \ref{t:upper}. It is already well-established that Gaussian decay in $v$ is propagated forward in time by the Landau equation \cite{cameron2017landau, henderson2017smoothing}, but for our purposes, the key is to obtain {\it quantitative} dependence of these Gaussian upper bounds on the $L^\infty$ norm of $f$ that is as sharp as possible. We accomplish this via a barrier argument with a barrier of the form $h(v) = K e^{-\mu |v|^2}$. Previous barrier arguments such as \cite{cameron2017landau, henderson2017smoothing} derived a contradiction at a first crossing point between $f$ and $h$ by writing
\[
Q(f,f) \leq Q(f,h) = \tr(\bar a^f D_v^2 h) + \bar c^f h,
\]
and bounding $\bar a^f$ and $\bar c^f$ using  some conditional upper bounds for $f$ like \eqref{e:old-condition}. By contrast, in our Proposition \ref{p:gaussian}, we get sharper estimates by also using $f\leq h$ in our bounds for the coefficients $\bar a^f$, $\bar c^f$. This kind of ``nonlinear barrier argument'' has been applied to the Boltzmann equation \cite{silvestre2016boltzmann, imbert2018decay} but is apparently new in the study of the Landau equation.

 \item 
 Finally, to remove the dependence on the energy density bound, we estimate $\int_{\R^3} |v|^2 f\dd v$ from above by interpolating between $\int_{\R^3} |v|^s f \dd v$ and the Gaussian upper bound.  This interpolation requires a non-obvious argument that previously appeared in \cite{S2018hardpotentials}. 

\end{enumerate}

 \subsection{Notation} 
 
We sometimes use the shorthand $z = (t,x,v) \in \R^7$,  as well as the notation $\langle a\rangle = \sqrt{1+|a|^2}$ for any vector or scalar $a$.

In this paper, all solutions to the Landau equation \eqref{e:landau} are understood to be solutions in the sense of Definition \ref{d:solution}. Whenever the time/space domain $A\subset \R^4$ is not specified, it is understood to be $[0,T]\times\R^3$ as in the statement of Theorem \ref{t:upper}.

To state local estimates, it is convenient to use kinetic cylinders of the form
\[
Q_r(z_0) =  \{(t,x,v) \in \R^7: t_0-r^2 < t \leq t_0,  |x-x_0 - (t-t_0) v_0|< r^3, |v-v_0| < r\} ,
\]
where $r>0$. We also write $Q_r = Q_r(0)$.  These cylinders are well-adapted to the scaling and translation symmetries of certain kinetic equations. More specifically, if $u$ solves the linear equation 
\[
\partial_t u + v\cdot \nabla_x u = \Delta_v u
\] 
in $Q_r(z_0)$, then $u(t_0 + r^2 t, x_0 + r^3 x +r^2 t v_0, v_0 + r v)$ solves the same equation in $Q_1$. This makes $Q_r(z_0)$ a natural domain on which to study local estimates for more general linear kinetic equation such as \eqref{e:linear}.
 
 The notation $A\lesssim B$ means $A\leq C B$ for a constant $C>0$ depending only on the quantities stated in the given lemma or theorem. The notation $A \approx B$ means $A\lesssim B$ and $B\lesssim A$.

 \subsection{Outline of the paper} In Section \ref{s:prelim}, we review some bounds on the coefficients $\bar a^f$, $\bar b^f$, and $\bar c^f$, as well as some known results on the spreading of positivity. Section \ref{s:upper} proves a local $L^\infty$ estimate via Moser iteration, Section \ref{s:improved} derives the global $L^\infty$ estimate of Theorem \ref{t:other-upper}, and Section \ref{s:gaussian} establishes quantitative Gaussian upper bounds. Finally, Section \ref{s:energy} explains how to  remove any dependence on the energy density bound in our results.

\section{Preliminaries}\label{s:prelim}

\subsection{Coefficient bounds}

The following lemma gives upper bounds for the coefficients $\bar a^f$, $\bar b^f$, $\bar c^f$, under the assumption that $L^1_v$ and $L^{p+\delta}_v$ norms of $f$ are bounded, where $p = 3/(5+\gamma)$. The proof is standard, but we need to track the dependence on the $L^\infty$ norm of $f$ precisely.

\begin{lemma}\label{l:coeffs}
Let $f:\R^3 \to \R$ belong to the space $L^1(\R^3)\cap L^\infty(\R^3)$. Let $p = \dfrac{3}{5+\gamma}$, and let $\delta>0$ be an arbitrary small number. 

Then the coefficients $\bar a^f$, $\bar b^f$, and $\bar c^f$ defined in \eqref{e:abc} satisfy, for all $v\in \R^3$,
\[
\begin{split}
 |\bar a^f(v)| &\leq C,\\
|\bar b^f(v)| &\leq C \|f\|_{L^\infty(\R^3)}^{1-p(\gamma+4)/3},\\
|\bar c^f(v)| &\leq C \|f\|_{L^\infty(\R^3)}^{1-p(\gamma+3)/3},
 \end{split}
 \]
for a constant $C>0$ depending only on $\gamma$, $\delta$, and the $L^{p+\delta}(\R^3)$ and $L^1(\R^3)$ norms of $f$. 
\end{lemma}

\begin{proof}
The bound for $\bar a^f$ follows from $|(I - |z|^{-2}z\otimes z)| \leq 1$ and the standard convolution estimate
\[
(|v|^{\gamma+2} \ast f)(v) \leq C \|f\|_{L^{p+\delta}(\R^3)}^{-(p+\delta)'(\gamma+2)/3} \|f\|_{L^1(\R^3)}^{1 + (p+\delta)'(\gamma+2)/3}.
\]
where $(p+\delta)' = (p+\delta)/(p+\delta - 1)$. The bounds for $\bar b^f$ and $\bar c^f$ follow from
\[
(|v|^\sigma \ast f)(v) \leq C \|f\|_{L^\infty(\R^3)}^{1-p(\sigma+3)/3} \|f\|_{L^p(\R^3)}^{p(\sigma+3)/3}, \quad \text{for $\sigma< - 3\left(1- \frac 1 p\right)$},
\]
which holds for both $\sigma = \gamma+1$ and $\sigma = \gamma$, since $p< 3/(3+\gamma)$. Note that we can absorb $\|f\|_{L^p(\R^3)}$ into the constant $C$ in the statement of the lemma, by interpolation.
\end{proof}

Next, we have a lower ellipticity bound for the matrix $\bar a^f$:

\begin{lemma}{\cite[Lemma 4.3]{HST2018landau}}\label{l:a-lower}
Let $f:\R^3\to [0,\infty)$ be an integrable function such that
\[
f(v) \geq \ell, \quad v\in B_\rho(0),
\]
for some $\ell, \rho>0$. Then the matrix $\bar a^f$ defined in \eqref{e:abc} satisfies
\begin{equation}\label{e:a-lower}
e\cdot( \bar a^f e) \geq c_a
\begin{cases}
\vv^\gamma, &e \in \mathbb S^2,\\
\vv^{\gamma+2}, &e\cdot v = 0.
\end{cases}
\end{equation}
The constant $c_a>0$ depends only on $\gamma$, $\ell$, and $\rho$. 
\end{lemma}

\subsection{Pointwise lower bounds}

Lower bounds for the solution $f$ will be combined with Lemma \ref{l:a-lower} to conclude coercivity of the matrix $\bar a^f$, which is essential for the smoothing properties of the equation. 



\begin{lemma}\label{l:lower-bound}
Let $f:[0,T]\times\R^6\to [0,\infty)$ be a solution of the Landau equation, satisfying $M_f(t,x) \leq M_0$ and $\|f(t,x,\cdot)\|_{L^{p+\delta}_v(\R^3)} \leq P_0$, uniformly in $(t,x)$, where $p = \dfrac 3 {5+\gamma}$.  Assume that the initial data is uniformly $\alpha$-H\"older continuous on $\R^6$ for some $\alpha \in (0,1)$, and satisfies the lower bound
\[
f_{\rm in}(x,v) \geq \ell, \quad x \in \R^3, v \in B_\rho(0),
\]
for some $\ell, \rho>0$. Furthermore, assume that the initial data satisfies the decay estimate $f_{\rm in}(x,v) \leq K \vv^{-q}$ for some $q>0$ sufficiently large, depending on $\alpha$ and $\gamma$. 

Then $f$ satisfies lower bounds of the form
\[
f(t,x,v) \geq \ell', \quad x\in \R^3, v\in B_{\rho/2}(0),
\]
where the constant $\ell'>0$ depends only on $\gamma$, $\ell$, $\rho$, $\delta$, $T$, $M_0$, and $P_0$. 
\end{lemma}

The exact statement of this lemma is not in the literature, but can be justified by combining known facts, in the following way: propagation of lower bounds was established in \cite[Theorem 1.3]{HST2018landau} for solutions with initial data lying in a fourth-order Sobolev space. (A more precise statement of the same result can be found in \cite[Lemma 2.5]{HST2019rough}.) Next, the existence theorem \cite[Theorem 1.2]{HST2019rough} constructed a local solution of \eqref{e:landau} given any bounded, measurable initial data with polynomial decay in $v$ and that is strictly positive in some small ball. As pointed out in \cite{HST2019rough} (see the remark after Lemma 2.5 in \cite{HST2019rough}), this constructed solution also satisfies the lower bounds for positive times, i.e. the conclusion of \cite[Lemma 2.5]{HST2019rough}. The {\it a priori} solution in the statement of our Lemma \ref{l:lower-bound} is also entitled to these lower bounds because of the uniqueness theorem \cite[Theorem 1.4]{HST2019rough}. The assumptions of polynomial decay and H\"older continuity for $f_{\rm in}$ in Lemma \ref{l:lower-bound} are only needed in order to apply this uniqueness theorem. 

Note that our classical solutions, as defined in Definition \ref{d:solution} are locally $C^1_{t,x} C^2_v$, which implies in particular that the initial data is pointwise $C^1$ in $x$ and $C^2$ in $v$.   Therefore, Lemma \ref{l:lower-bound} applies to the solutions considered in our main results.

\section{Local $L^\infty$ estimate}\label{s:upper}

In this section, we consider a solution $f$ to the Landau equation on $(-1,0]\times B_1 \times \R^3_v$. We assume that the matrix $\bar a^f$ satisfies some lower ellipticity bound
\[
e\cdot (\bar a^f(t,x,v) e) \geq \lambda, \quad (t,x,v) \in Q_1, e\in \mathbb S^2,
\]
for some $\lambda>0$. Later, we will recenter the estimate around any arbitrary point $(t_0,x_0,v_0)$, and calculate $\lambda$ depending on $v_0$ via Lemma \ref{l:a-lower}.

As discussed in the introduction, the argument of this sectionis a modification of the work in \cite{golse2016}. The estimate we obtain (Proposition \ref{p:inf}) is in a form that is convenient to apply scaling techniques and eventually remove any dependence on $\|f\|_{L^\infty}$. 

\begin{lemma}\label{l:subsolution}
Let $f\geq 0$ be a classical solution of the Landau equation in $(-1,0]\times B_1\times \R^3_v$, and let $\chi(t,x,v)$ be any smooth, compactly supported function in $Q_1$. Then for any $q\geq 1$, one has
\[
(\partial_t + v\cdot \nabla_x) (\chi f^q) \leq \nabla_v \cdot (\bar a^f\nabla_v (\chi f^q)) + H_0 + \nabla_v \cdot H_1,
\]
with 
\[
\begin{split}
H_0 &= f^q [(\partial_t + v\cdot \nabla_x)\chi + \nabla_v\cdot (\bar a^f\nabla_v \chi) - \nabla_v\cdot(\chi \bar b^f) + q \chi \bar c^f],\\
H_1 &= f^q [-2\bar a^f\nabla_v \chi + \chi \bar b^f].
\end{split}
\]
\end{lemma}

\begin{proof}
The proof is a direct calculation involving several applications of the product rule. In more detail, using the equation \eqref{e:landau}, we have
\begin{equation}\label{e:abc-eqn}
\begin{split}
(\partial_t + v\cdot \nabla_x)(\chi f^q) 
&= f^q (\partial_t + v\cdot \nabla_x) \chi + q f^{q-1} \chi [\nabla_v\cdot(\bar a^f \nabla_v f) + \bar b^f\cdot \nabla_v f + \bar c^f f].
\end{split}
\end{equation}
For the term on the right involving $\bar a^f$, we have
\[
\begin{split}
q f^{q-1} \chi \nabla_v\cdot(\bar a^f \nabla_v f) &= q\nabla_v\cdot (f^{q-1} \chi \bar a^f \nabla_v f) - q\nabla_v( f^{q-1}\chi)\cdot(\bar a^f \nabla_v f)\\
&= \nabla_v\cdot (\chi \bar a^f \nabla_v(f^q)) - q\chi \nabla_v(f^{q-1})\cdot (\bar a^f \nabla_v f) - qf^{q-1} \nabla_v \chi \cdot (\bar a^f \nabla_v f)\\
&= \nabla_v\cdot (\bar a^f \nabla_v(\chi f^q)) - \nabla_v \cdot ( f^q \bar a^f \nabla_v \chi)\\
&\quad - q(q-1) \chi f^{q-2}\nabla_v f \cdot (\bar a^f \nabla_v f) - \nabla_v \chi\cdot (\bar a^f\nabla_v f^q)\\
&\leq \nabla_v\cdot (\bar a^f \nabla_v(\chi f^q)) - \nabla_v \cdot ( f^q \bar a^f \nabla_v \chi) - \nabla_v \chi\cdot (\bar a^f\nabla_v f^q),
\end{split}
\]
by the positive-definiteness of $\bar a^f$. Applying the product rule again in the last term, we have
\[
q f^{q-1} \chi \nabla_v\cdot(\bar a^f \nabla_v f) \leq \nabla_v\cdot (\bar a^f \nabla_v(\chi f^q)) - 2\nabla_v \cdot ( f^q \bar a^f \nabla_v \chi) +f^q \nabla_v\cdot(\bar a^f\nabla_v \chi).
\]
For the $\bar b^f$ term in \eqref{e:abc-eqn}, we have
\[
\begin{split}
q f^{q-1} \chi \bar b^f\cdot \nabla_v f &= \chi \bar b^f\cdot \nabla_v (f^q) = \nabla_v\cdot(\chi  f^q \bar b^f) - f^q\nabla_v\cdot( \chi \bar b^f).
\end{split}
\]
After collecting terms, we obtain the statement of the lemma.
\end{proof}

\begin{lemma}\label{l:hypo}
With $f$, $\chi$, $H_0$, and $H_1$ as in Lemma \ref{l:subsolution}, the following inequality holds for any $q\geq 1$:
\[
\|\chi f^q\|_{L^{42/19}(Q_1)}^2 \leq C \left(\frac{1+\|\bar a^f\|_{L^\infty(Q_1)}^2} {\lambda^2}\right) \left( \|H_0\|_{L^2(Q_1)}^2 + \|H_1\|_{L^2(Q_1)}^2\right),
\]
for a universal constant $C>0$. In particular, $C>0$ is independent of $q$ and $\chi$. 
\end{lemma}

\begin{proof}
Let $g$ be the solution to
\begin{equation}\label{e:g-equation}
\begin{split}
(\partial_t +v\cdot\nabla_x) g &= \nabla_v \cdot (\bar a^f \nabla_v g) + H_0 + \nabla_v \cdot H_1, \quad (t,x,v) \in Q_1,\\
 g &= 0,  \quad \text{ $t=-1$ or $|x| = 1$ or $|v| = 1$}.
\end{split}
\end{equation}
Here, $H_0$ and $H_1$ are defined in terms of the function $f$.  By Lemma \ref{l:subsolution}, the function $h = g- \chi f^q$ satisfies $(\partial_t + v\cdot\nabla_x) h = \nabla_v \cdot (\bar a^f \nabla_v h)$. Integrating this equation against $h$ over the domain $Q_1$ and using the fact that $h = 0$ on the time slice $t=-1$ and the boundary $\{ |x| =1\} \cup \{|v| = 1\}$, we conclude $g \geq \chi f^q$ in $Q_1$.


Next, integrating equation \eqref{e:g-equation} against $g$ over $Q_1 \subset\R^{7}$, we obtain (writing $\dd z = \dd t \dd x \dd v$)
\[
\begin{split}
\frac 1 2  \int_{Q_1} \frac d {dt} g^2 \dd z   &\leq -   \lambda  \int_{Q_1} |\nabla_v g|^2\dd z + \int_{Q_1} g H_0 \dd z - \int_{Q_1} \nabla_v g \cdot H_1 \dd z\\
&\leq \int_{Q_1} g H_0 \dd z  -  \frac \lambda 2 \int_{Q_1} |\nabla_v g |^2 \dd z + \frac 1 {2\lambda} \int_{Q_1} |H_1|^2 \dd z .\\
\end{split}
\]
using the fact that $g=0$ on $\{|x| = 1\} \cup \{|v| = 1\}$. The left side of this inequality is nonnegative because
  $g=0$ on the time slice $\{t=-1\}$, and $g \geq \chi f^q \geq 0$ on the time slice $\{t=0\}$.  Letting $C_P$ be the constant in the Poincar\'e inequality $\int_{B_1} g^2 \dd v \leq C_P \int_{B_1} |\nabla_v g|^2 \dd v$, we have
\[
\begin{split}
0 &\leq  \frac{\lambda }{4 C_P} \int_{Q_1} g^2 \dd z + \frac {C_P } {4\lambda} \int_{Q_1} H_0^2 \dd z -  \frac \lambda 2 \int_{Q_1} |\nabla_v g |^2 \dd z + \frac 1 {2\lambda} \int_{Q_1} |H_1|^2 \dd z \\
&\leq  -  \frac \lambda 4 \int_{Q_1} |\nabla_v g |^2 \dd z + \frac {C_P } {4\lambda} \int_{Q_1} H_0^2 \dd z + \frac 1 {2\lambda} \int_{Q_1} |H_1|^2 \dd z .
\end{split}
\]
Collecting terms, we obtain
\begin{equation}\label{e:g-energy}
 \int_{Q_1} |\nabla_v g|^2 \dd z \leq C \|g\|_{L^2(Q_1)}^2 + \frac C{\lambda^2} \left(\|H_0\|_{L^2(Q_1)}^2 + \|H_1\|_{L^2(Q_1)}^2\right).
\end{equation}

Next, we apply the hypoelliptic estimate of Bouchut \cite[Theorem 1.3]{bouchut2002hypoelliptic} to $g$, yielding
\[
\begin{split}
\|D_t^{1/3}g\|_{L^2(Q_1)}^2 + \|D_x^{1/3}g\|_{L^2(Q_1)}^2  &\lesssim \|g\|_{L^2(Q_1)}^2 + \|\nabla_v g\|_{L^2(Q_1)} \| \vv H_0\|_{L^2(Q_1)}\\
&\quad + \|\nabla_v g\|_{L^2(Q_1)}^{4/3}\|\vv(H_1+\bar a^f \nabla_v g)\|_{L^2(Q_1)}^{2/3}\\
&\quad + \|\nabla_v g\|_{L^2(Q_1)} \|\vv(H_1+\bar a^f \nabla_v g)\|_{L^2(Q_1)}\\
&\lesssim \|g\|_{L^2(Q_1)}^2 + \|\nabla_v g\|_{L^2(Q_1)}^2\\ 
&\quad+ \| \vv H_0\|_{L^2(Q_1)}^2 + \|\vv(H_1+\bar a^f \nabla_v g)\|_{L^2(Q_1)}^2.
\end{split}
\]
By the Poincar\'e inequality, the term $\|g\|_{L^2(Q_1)}^2$ on the right can be absorbed into $\|\nabla_v g\|_{L^2(Q_1)}^2$. Adding $\|\nabla_v g\|_{L^2(Q_1)}^2$ to both sides and using $\vv\leq 1$ as well as the energy estimate \eqref{e:g-energy}, we obtain
\[
\begin{split}
\|D_t^{1/3}g\|_{L^2(Q_1)}^2 + &\|D_x^{1/3}g\|_{L^2(Q_1)}^2 + \|\nabla_v g\|_{L^2(Q_1)}^2 \\
 &\lesssim (1+\|\bar a^f\|_{L^\infty(Q_1)}^2)\|\nabla_v g\|_{L^2(Q_1)}^2 + \|H_0\|_{L^2(Q_1)}^2 + \|H_1\|_{L^2(Q_1)}^2 \\
 &\lesssim \frac{1+\|\bar a^f\|_{L^\infty(Q_1)}^2} {\lambda^2} \left( \|H_0\|_{L^2(Q_1)}^2 + \|H_1\|_{L^2(Q_1)}^2\right).
\end{split}
\]
We now apply the Sobolev embedding $H^{1/3}(\R^7)\subset L^{42/19}(\R^7)$ to obtain 
\[
\begin{split}
\|g\|_{L^{42/19}(Q_1)} \lesssim  \frac{1+\|\bar a^f\|_{L^\infty(Q_1)}^2} {\lambda^2} \left( \|H_0\|_{L^2(Q_1)}^2 + \|H_1\|_{L^2(Q_1)}^2\right).
\end{split}
\]
With the inequality $\chi f^q \leq g$, the proof is complete.
\end{proof}

\begin{lemma}\label{l:gain}
Let $f\geq 0$ be a solution of the Landau equation in  $(-1,0]\times B_1\times \R^3_v$. For any $0< r_0 < r_1< 1$ and $q\geq 1$, there holds
\[
\begin{split}
\|f\|_{L^{\sigma q}(Q_{r_0})}^{2q} &\leq C \left(\frac {1 + \|\bar a^f\|_{L^\infty(Q_1)}^2} { \lambda^2} \right) \left( 1 + \|\bar a^f\|_{L^\infty(Q_1)}^2 + \|\bar b^f\|_{L^\infty(Q_1)}^2 + \|\bar c^f\|_{L^\infty(Q_1)}^2\right)\\
&\quad \quad\times \left( (r_1-r_0)^{-4} + q^2\right) \|f\|_{L^{2q}(Q_{r_1})}^{2q},
\end{split}
\]
with $\sigma = 42/19$,  and $C>0$ a constant depending only on $\gamma$.
\end{lemma}

\begin{proof}
First, we simplify $H_0$ using the relationships
\[
\sum_j \partial_{v_j}\bar a_{ij}^f = -\bar b_i^f, \quad \nabla_v\cdot \bar b^f = -\bar c^f,
\]
giving
\[
H_0 =f^q[(\partial_t + v\cdot\nabla_x)\chi + \tr(\bar a^f D_v^2 \chi)  - 2\bar b^f\cdot\nabla_v\chi + (q+1) \chi \bar c^f.
\]
Next, we choose $\chi \in C_0^\infty(Q_1)$ so that $\chi = 1$ in $Q_{r_0}$ and $\chi = 0$ outside $Q_{r_1}$. Such a $\chi$ can be chosen so that
\[
|(\partial_t + v\cdot \nabla_x)\chi |\lesssim (r_1-r_0)^{-2}, \quad |\nabla_v\chi|\lesssim (r_1-r_0)^{-1}, \quad |D_v^2 \chi| \lesssim (r_1-r_0)^{-2}.
\]
 (The existence of a cutoff $\chi$ with the desired properties follows from the geometry of kinetic cylinders. In more detail, one may construct a cutoff $\widetilde \chi$ that equals 1 in $Q_{r_0/(r_1-r_0)}$ and zero outside $Q_{r_1/(r_1-r_0)}$, and $\widetilde \chi$ can be chosen with $t$, $x$, and $v$ derivatives all bounded by a universal constant. Next, let $\chi(t,x,v) = \widetilde \chi(\rho^2 t, \rho^3 x, \rho v)$ with $\rho = r_1 - r_0$. By direct calculation, $\chi$ satisfies the stated bounds on its derivatives.)

With this choice of $\chi$, note that $H_0$ and $H_1$ are zero outside $Q_{r_1}$. We bound $H_0$ and $H_1$ as follows, using $r_1 - r_0 < 1$:
\[
\begin{split}
\|H_0\|_{L^2(Q_1)} &\lesssim \|f^q\|_{L^2(Q_{r_1})} \left[  (r_1 - r_0)^{-2} \left( 1 + \|\bar a^f\|_{L^\infty(Q_1)} + \|\bar b^f\|_{L^\infty(Q_1)}\right) + q\|\bar c^f\|_{L^\infty(Q_1)} \right]\\
&\lesssim \|f\|_{L^{2q}(Q_{r_1})}^q \left( 1 + \|\bar a^f\|_{L^\infty(Q_1)} + \|\bar b^f\|_{L^\infty(Q_1)} + \|\bar c^f\|_{L^\infty(Q_1)}\right) \left( (r_1 - r_0)^{-2} + q\right),
\end{split}
\]
and
\[
\begin{split}
\|H_1\|_{L^2(Q_1)} &\lesssim  \|f\|_{L^{2q}(Q_{r_1})}^q \left( \|\bar a^f\|_{L^\infty(Q_1)} (r_1 - r_0)^{-1} + \|\bar b^f\|_{L^\infty(Q_1)}\right)\\
&\lesssim \|f\|_{L^{2q}(Q_{r_1})}^q (r_1 - r_0)^{-1} \left(1 + \|\bar a^f\|_{L^\infty(Q_1)} + \|\bar b^f\|_{L^\infty(Q_1)}\right).
\end{split}
\]
Combining these estimates for $H_0$ and $H_1$ with Lemma \ref{l:hypo} yields the conclusion of the lemma. 
\end{proof}

Now we use Lemma \ref{l:gain} and a classical Moser iteration procedure to prove a local $L^\infty$ estimate for $f$:

\begin{proposition}\label{p:inf}
Let $f:(-1,0]\times\R^3\times\R^3\to [0,\infty)$ solve the Landau equation \eqref{e:landau} in   $(-1,0]\times B_1\times \R^3_v$. Assume that the coefficients $\bar a^f$, $\bar b^f$, and $\bar c^f$ are essentially bounded in $Q_1$, and that the matrix $\bar a^f$ satisfies the lower ellipticity bound
\[
e\cdot (\bar a^f(t,x,v) e) \geq \lambda, \quad (t,x,v) \in Q_1, e\in \mathbb S^2.
\]
 Then 
\[
\|f\|_{L^\infty(Q_{1/2})} \leq C \left(\frac {1 + \|\bar a^f\|_{L^\infty(Q_1)}^2} { \lambda^2} \right)^{19/4} \left( 1 + \|\bar a^f\|_{L^\infty(Q_1)}^2 + \|\bar b^f\|_{L^\infty(Q_1)}^2 + \|\bar c^f\|_{L^\infty(Q_1)}^2\right)^{19/4} \|f\|_{L^2(Q_1)},
\]
where the constant $C$ depends only on $\gamma$. 
\end{proposition}

\begin{proof}
Define the radii $r_i$ and exponents $q_i$ by
\[
r_i := \frac 1 2 + \left(\frac 1 2\right)^i, \quad q_i = \left(\frac \sigma 2\right)^i = \left(\frac {21}{19}\right)^i, \quad i = 1,2,\ldots
\]
Since Lemma \ref{l:gain} holds for any $q>1$ and any concentric cylinders $Q_{r_0} \subset Q_{r_1}$ with $r_0<r_1 <1 $, we have for each $i$ the inequality
\begin{equation}\label{e:ineq}
\begin{split}
\|f\|_{L^{\sigma q_i}(Q_{r_{i+1}})} &\leq  K_f^{1/(2q_i)}\left( (r_i - r_{i+1})^{-4} + q_i^2\right)^{1/(2q_i)} \|f\|_{L^{2q_i}(Q_{r_i})},
\end{split}
\end{equation}
with
\[
K_f := C\left(\frac {1 + \|\bar a^f\|_{L^\infty(Q_1)}^2} { \lambda^2} \right) \left( 1 + \|\bar a^f\|_{L^\infty(Q_1)}^2 + \|\bar b^f\|_{L^\infty(Q_1)}^2 + \|\bar c^f\|_{L^\infty(Q_1)}^2\right).
\]
Note that $q_i^2 = (21/19)^{2i} \leq (16)^{i+1} = (r_{i+1} - r_i)^{-4}$, so we can rewrite \eqref{e:ineq} as 
\[
\begin{split}
 \|f\|_{L^{2q_{i+1}}(Q_{r_{i+1}})} = \|f\|_{L^{\sigma q_i}(Q_{r_{i+1}})} &\leq  K_f^{1/(2q_i)}2^{(2i+5/2)/q_i} \|f\|_{L^{2q_i}(Q_{r_i})}.
\end{split}
\]
 Iterating from $i=1,2,\ldots$, we obtain the desired upper bound
 \[
 \|f\|_{L^\infty(Q_{1/2})} \leq K_f^{19/4} 2^{893/4} \|f\|_{L^2(Q_{1})},
 \]
 since
 \[
 \sum_{i=1}^\infty \frac 1 {2q_i} = \frac {19} 4 \quad \text{and} \quad \sum_{i=1}^\infty \frac{2i+5/2}{q_i} = \frac{893} 4.
 \]
\end{proof}

\section{Global $L^\infty$ estimate}\label{s:improved}

In this section, we improve the local $L^\infty$ bound of Proposition \ref{p:inf} via scaling techniques, and use them to prove the global upper bounds of Theorem \ref{t:other-upper}.

\begin{lemma}\label{l:Qr-est-scaling}
Let $f:[0,T]\times \R^6\to [0,\infty)$ be a solution of the Landau equation \eqref{e:landau} satisfying the hypotheses of Theorem \ref{t:other-upper}, as well as
\[
L_0 := \|f\|_{L^\infty([0,T]\times \R^6)} \geq 1.
\]
 Then, with $p = \dfrac 3{5+\gamma}$, any small $\delta>0$, and $r = L_0^{-p/3} \min\{1, \sqrt {t_0/2}\}$, there holds
\[
f(t_0,x_0,v_0) \leq C \langle v_0\rangle^{-19\gamma/2} (1+t_0^{-3/4}) L_0^{p/2} \|f\|_{L^\infty(Q_r(z_0))}^{(2-p-\delta)/2} \|f\|_{L^\infty_{t,x} L^{p+\delta}_v(Q_r(z_0))}^{(p+\delta)/2},
\]
where $C$ depends on $\gamma$, $\delta$, $\ell$, $\rho$, as well as 
\[
\sup_{t,x} \left( \int_{\R^3} f(t,x,v) \dd v + \|f(t,x,\cdot)\|_{L^{p+\delta}(\R^3)} \right).
\]
\end{lemma}

\begin{proof}
With $f$ as in the statement, define the rescaled solution 
\[
f_r(t,x,v) = r^{5+\gamma} f(t_0+r^2 t, x_0 + r^3x + r^2 t v_0, v_0 + rv).
\]
By direct calculation, $f_r$ is also a solution to the Landau equation. Applying the $L^\infty$ estimate of Proposition \ref{p:inf} to $f_r$, we have
\[
\begin{split}
f(t_0,x_0,v_0) &\leq \|f\|_{L^\infty(Q_{r/2}(z_0))}\\
 &= \frac{\|f_r\|_{L^\infty(Q_{1/2})}}{r^{5+\gamma}}\\
 &\lesssim  \left( \frac {\|\bar a^{f_r}\|_{L^\infty(Q_1)}} {\lambda[f_r]}(1+ \|\bar a^{f_r}\|_{L^\infty(Q_1)} + \|\bar b^{f_r}\|_{L^\infty(Q_1)} + \|\bar c^{f_r}\|_{L^\infty(Q_1)})\right)^{19/2} \frac {\|f_r\|_{L^2(Q_1)}} {r^{5+\gamma}} .
\end{split}
\]
Note that the coefficients appearing in this right-hand side are defined in terms of $f_r$, and $\lambda[f_r]$ is the lower ellipticity constant corresponding to $\bar a^{f_r}$. Calculating these coefficients in terms of $f$, we have
\[
\begin{split}
\bar a^{f_r}(t,x,v) &= \bar a^f(t_0+r^2 t, x_0 + r^3x + r^2 t v_0, v_0 + rv),\\
\bar b^{f_r}(t,x,v) &= r \bar b^f(t_0+r^2 t, x_0 + r^3x + r^2 t v_0, v_0 + rv),\\
\bar c^{f_r}(t,x,v) &= r^2 \bar c^f(t_0+r^2 t, x_0 + r^3x + r^2 t v_0, v_0 + rv),
\end{split}
\]
and from Lemma \ref{l:a-lower}, $e\cdot (a^{f_r}(t,x,v) e) \geq \lambda[f_r] \approx \langle v_0 \rangle^\gamma$ for all $e\in \mathbb S^2$. This yields
\[
\begin{split}
&f(t_0,x_0,v_0)\\
 &\lesssim \left( \frac {\|\bar a^{f}\|_{L^\infty(Q_r(z_0))}} {\langle v_0\rangle^\gamma}(1+ \|\bar a^{f}\|_{L^\infty(Q_r(z_0))} + r \|\bar b^{f}\|_{L^\infty(Q_r(z_0))} + r^2\|\bar c^{f}\|_{L^\infty(Q_r(z_0))})\right)^{19/2} \frac {\|f_r\|_{L^2(Q_1)}} {r^{5+\gamma}}.
\end{split}
\]

We analyze the $L^2$ norm on the right as follows:
\[
\frac{\|f_r\|_{L^2(Q_1)}}{r^{5+\gamma}} \leq \frac{\|f_r\|_{L^\infty_{t,x} L^2_v(Q_1)}}{r^{5+\gamma}} = r^{-3/2} \|f\|_{L^\infty_{t,x} L^2_v(Q_r(z_0))},
\]
from the definition of $f_r$. 
Recalling the notation $L_0 = \|f\|_{L^\infty([0,T]\times\R^6)}$, and incorporating the coefficient estimates from Lemma \ref{l:coeffs}, we have
\[
f(t_0,x_0,v_0) \lesssim  \langle v_0\rangle^{-19\gamma/2} \left( 1 + rL_0^{1-p(\gamma+4)/3} + r^2 L_0^{1-p(\gamma+3)/3}\right)^{19/2} r^{-3/2} \|f\|_{L^\infty_{t,x} L^2_v(Q_r(z_0))}.
\]
By our choice of $r$ in the statement of the lemma, the terms inside the parentheses balance,  
and the estimate becomes
\begin{equation}\label{e:19}
\begin{split}
f(t_0,x_0,v_0)
& \lesssim \langle v_0\rangle^{-19\gamma/2} \left(1 +  L_0^{1-p(\gamma+5)/3}\right)^{19/2} L_0^{p/2}(1+ t_0^{-3/4}) \|f\|_{L^\infty_{t,x} L^2_v(Q_r(z_0))}\\
& \lesssim \langle v_0\rangle^{-19\gamma/2} (1+t_0^{-3/4}) L_0^{p/2} \|f\|_{L^\infty_{t,x} L^2_v(Q_r)},
\end{split}
\end{equation}
where we used $p = \dfrac 3 {5+\gamma}$.  

Next, interpolating between $L^\infty$ and $L^{p+\delta}$ (since we can choose $\delta$ small enough that $p+\delta < 2$), we have
\[
 \|f\|_{L^\infty_{t,x}L^2_v(Q_r(z_0))} \leq \|f\|_{L^\infty(Q_r(z_0))}^{(2-p-\delta)/2} \|f\|_{L^\infty_{t,x} L^{p+\delta}_v(Q_r(z_0))}^{(p+\delta)/2}.
 \]
Using this in \eqref{e:19}, the proof is complete. 
\end{proof}

The growing power of $\langle v_0\rangle$ in Lemma \ref{l:Qr-est-scaling} will be dealt with in two different ways: by absorbing the growth into a weighted conditional norm of $f$ in the proof of Theorem \ref{t:other-upper}, and by propagating decay forward from the initial data in Theorem \ref{t:upper}.

\begin{proof}[Proof of Theorem \ref{t:other-upper}]

First, for small values of time, the solution $f$ is bounded in $L^\infty$ by some value depending only on the initial data. This can be seen, for example, by applying the existence/uniqueness theorem of \cite{HST2019rough}: for some $T_*>0$ depending only on $\|\vv^5 f_{\rm in}\|_{L^\infty(\R^6)}$, one has $\|f(t)\|_{L^\infty(\R^6)} \leq \|\vv^5 f(t)\|_{L^\infty(\R^6)} \leq 2 \|\vv^5 f_{\rm in}\|_{L^\infty(\R^6)} = :L_*$ whenever $t\leq T_*$.

Next, let $z_0 = (t_0,x_0,v_0)\in [0,T]\times\R^6$ be chosen so that
\[
f(t_0,x_0,v_0) \geq \frac 1 2 \|f\|_{L^\infty([0,T]\times\R^6)}.
\]
We may assume  $t_0 > T_*$, since otherwise, $f$ is bounded by $L_*$ in all of $[0,T]\times\R^6$ and there is nothing left to show.

Applying Lemma \ref{l:Qr-est-scaling} at this $z_0$, we obtain
\[
f(t_0,x_0,v_0) \leq \langle v_0 \rangle^{-19\gamma/2}(1+t_0^{-3/4}) L_0^{1-\delta/2} \|f\|_{L^\infty_{t,x} L^{p+\delta}_v(Q_r(z_0))}^{(p+\delta)/2}.
\]
Since $t_0> T_*$, and $T_*$ depends only on the initial data, we absorb $(1+t_0^{-3/4})$ into the implied constant. Next, we absorb the factor $\langle v_0 \rangle^{-19\gamma/2}$ into the norm of $f$:
\[
\begin{split}
f(t_0,x_0,v_0) &\lesssim L_0^{1-\delta/2} \|\vv^{-19\gamma/(p+\delta)} f\|_{L^\infty_{t,x} L^{p+\delta}_v(Q_r(z_0))}^{(p+\delta)/2}\\
&\lesssim L_0^{1-\delta/2},
\end{split}
\]
with implied constant depending on the weighted $L^{p+\delta}_{-19\gamma/(p+\delta)}$ bound for $f$. Since $(t_0,x_0,v_0)$ was chosen so that $f(t_0,x_0,v_0) \geq L_0/2$, this implies $L_0^{\delta/2}$ is bounded above by a constant depending only on $\delta$, the initial data, and the $L^\infty_{t,x} L^1_v$ and $L^\infty_{t.x} (L^p_{-19\gamma/(p+\delta)})_v$ norms of $f$. This concludes the proof.
\end{proof}

\section{Gaussian bounds}\label{s:gaussian}

This section establishes Gaussian decay estimates in $v$ for the solution $f$. These estimates will be needed to remove the velocity weights from the conditional assumption on $f$ for positive times, in our upper bounds. 

\begin{proposition}\label{p:gaussian}
Let $f:[0,T]\times\R^6 \to [0,\infty)$ be a solution of the Landau equation satisfying \eqref{e:cond}, and assume the initial data satisfies
\[
f_{\rm in}(x,v) \leq C_0 e^{-\mu' |v|^2},
\]
for some $C_0, \mu'>0$, as well as the lower bounds
\[
f_{\rm in}(x,v) \geq \ell, \quad x\in \R^3, v\in B_\rho(0).
\]
 for some $\ell, \rho>0$. Assume further that $f$ is bounded and that
 \[
 \int_{\R^3}|v|^2 f(t,x,v) \dd v \leq E_0, \quad (t,x)\in [0,T]\times\R^3.
  \]
Then there exist $c_0>0$ and $C_1>1$, depending on $\gamma$, $C_0$, $\ell$, $\rho$, and the constants in \eqref{e:cond},  such that for 
\[
K = 2\max\{ C_0, C_1\|f\|_{L^\infty([0,T]\times\R^6)}\}
\]
and
\begin{equation}\label{e:mu-def}
\mu = \min\left\{\frac {\mu'} 2 , \frac {c_0} {E_0}, \frac {1} {33\log(K/c_0)}\right\},
\end{equation}
 the upper bound 
\[
f(t,x,v) \leq K e^{-\mu |v|^2},
\]
holds for all $(t,x,v) \in [0,T]\times\R^6$. 
\end{proposition}
\begin{proof}

First, let us reduce to the case where $f$ is periodic in the $x$ variable and decays rapidly in $v$ (in a qualitative sense). These properties will be needed when we find a first crossing point in our barrier argument. For large $R>0$, let $\zeta_R:\R^3\to [0,1]$ be a smooth function that equals 1 in $B_{R/2}(0)$ and 0 outside $B_{R}(0)$. Let $\mathbb T_R^3\supset B_R(0)$ be the torus of side length $2R$, and let $f^R$ be the solution to the Landau equation \eqref{e:landau} with initial data
\[
f_{\rm in}^R(x,v) = \zeta_R(x) \zeta(v) f_{\rm in}(x,v), \quad (x,v) \in \mathbb T^3_R \times \R^3.
\]
By the existence theorem \cite[Theorem 1.2]{HST2019rough} these solutions exist on $[0,T_0]\times\mathbb T_R^3\times\R^3$ for some $T_0\leq T$ depending only on $\|\vv^5 f_{\rm in}^R\|_{L^\infty(\R^6)}$, which is bounded independently of $R$. Furthermore, fixing any $t_1>0$ and any bounded domain $\Omega \subset [t_1,T_0]\times\R^6$, the lower bounds of Lemma \ref{l:lower-bound} and the smoothing estimates of \cite[Theorem 1.2]{henderson2017smoothing} imply the family $\{f^R\}_{R\geq R_0}$ is precompact in $C^k(\Omega)$ for any $k$, for some $R_0>0$ depending on $\Omega$. Therefore, a sequence $R_j\to \infty$ of $f^R$ (extended by periodicity in $x$) converges locally uniformly on $[0,T_0]\times\R^6$ to a limit, which has initial data $f_{\rm in}$ and therefore must equal $f$ by the uniqueness theorem \cite[Theorem 1.4]{HST2019rough}.

The initial data $f_{\rm in}^R$ is compactly supported, so it satisfies Gaussian decay in $v$ for any rate $\mu>0$. By \cite[Theorem 3.4]{henderson2017smoothing}, these upper bounds are propagated to the time interval $[0,T_0]$: 
\[
f(t,x,v) \leq K_{\mu,R,T_0} e^{-\mu|v|^2},
\]
for some constant $K_{\mu,R,T_0}>0$. This decay estimate will be used to obtain a first crossing point, but it will not be used quantitatively. 

It will suffice to prove the conclusion of this proposition for $f^R$, with constants independent of $R$, up to time $T_0$. Indeed, the upper bound $f^R(t,x,v) \leq K e^{-\mu |v|^2}$ and lower bounds of Lemma \ref{l:lower-bound} imply the solution can be extended to a larger time interval by re-applying the existence theorem \cite[Theorem 1.2]{HST2019rough}, and this argument can be repeated until $f^R$ exists on the same time interval as $f$. The conclusion of the lemma can then be transferred to $f$ by taking the pointwise limit of $f^R$. For simplicity, we assume $T_0=T$ and omit the dependence on $R$ for the rest of the proof.

Throughout this proof, we use the shorthand $\|f\|_{L^\infty} = \|f\|_{L^\infty([0,T]\times\R^6)}$. 

Define the barrier
\[
h(v) = K e^{-\mu|v|^2},
\]
with $K$ and $\mu$ as in the statement of the lemma. By construction, the inequality $f<h$ holds at $t=0$. If $f< h$ does not hold in all of $[0,T]\times\R^6$, we claim there is a point $(t_0,x_0,v_0)$ with $t_0>0$ where $f=h$ for the first time. As discussed at the beginning of this proof, we can assume $f$ is spatially periodic and decays faster than any Gaussian, so continuity in time guarantees the existence of such a point.

%

At the crossing point, since $h$ is constant in $t$ and $x$, we have
\[
\partial_t f \geq 0, \quad \nabla_x f = 0, \quad D^2_v f \leq D_v^2 h,
\]
as well as 
\begin{equation}\label{e:less-than-barrier}
f(t_0,x_0,v) \leq h(v), \quad v\in \R^3.
\end{equation}
From the equation, we then have
\begin{equation}\label{e:crossing-eqn}
0 \leq (\partial_t + v\cdot\nabla_x) f = \tr(\bar a^f D_v^2 f) + f^2 \leq \tr(\bar a^f D_v^2 h) +\bar c^f h,
\end{equation}
at $(t_0,x_0,v_0)$, by the positive-definiteness of $\bar a^f$. 

Our goal is to show the right side of \eqref{e:crossing-eqn} is negative. We begin by bounding the term $\tr(\bar a^f D_v^2 h)$ from above by a negative quantity. To to this, we would like to use the anisotropic upper and lower bounds for the quadratic form $e\mapsto e\cdot (\bar a^f e)$ given by Lemma \ref{l:a-lower}, so we write $D^2 h(v)$ as a sum of two terms, the first acting on vectors parallel to $v$, and the second acting on vectors perpendicular to $v$. By direct calculation,
\[
D_v^2 h(v) = 2\mu h(2\mu v\otimes v - I) = 2\mu h( (2\mu - |v|^{-2}) v\otimes v - (I - |v|^{-2} v\otimes v)).
\]
Using the positive definiteness of $\bar a^f$, we then have
\begin{equation}\label{e:d2term}
\begin{split}
\tr(\bar a^f D_v^2 h) &=2\mu h\left[ (2\mu - |v|^{-2})\tr(\bar a^f v\otimes v) - \tr(\bar a^f(I - v\otimes v/|v|^2))\right]\\
&\leq 2\mu h\left[ 2\mu\, \tr(\bar a^f v\otimes v) - \tr(\bar a^f(I - v\otimes v/|v|^2))\right]
\end{split}
\end{equation}
 By direct calculation (see, e.g. \cite[Lemma 2.1]{cameron2017landau})
\[
\tr\left[\Pi(v-z) (v\otimes v)\right] = \tr\left[\left(I - \frac{(v-z)\otimes (v-z)}{|v-z|^2} \right)(v\otimes v)\right] = |v|^2 |z|^2\sin^2 \theta_{z,v} |v-z|^{-2},
\]
where $\theta_{v,z}$ is the angle between $v$ and $z$. Therefore,
\[
\begin{split}
\tr (\bar a^f v\otimes v) 
&= 
 a_\gamma \int_{\R^3} \Pi(v-z) |v-z|^{\gamma+2} f(z) \dd z\\
 &= a_\gamma |v|^2 \int_{\R^3} |z|^2 \sin^2 \theta_{v,z} |v-z|^\gamma f(z) \dd z.
\end{split}
\]
To bound this integral (evaluated at $v=v_0$) from above, when $z$ is close to $v_0$, we use the Gaussian upper bound $f\leq h$ as well as the bound on the mass of $f$. In more detail, let $r = |v_0|/2$. When $z\in B_r(v_0)$, one has~$|v_0|\approx |z|$, $\sin\theta_{v_0,z} \leq |v_0-z|/|v_0|$, and $f(z) \leq h(z) \leq e^{-|v_0|^2/4}$, which implies
\[
\begin{split}
|v|^2\int_{B_r(v_0)}|z|^2 &\sin^2\theta_{v_0,z}  |v_0-z|^{\gamma} f(z) \dd z\\
 &\leq K^{1/2} e^{-\mu |v_0|^2/8} |v|^2 \int_{B_r(v_0)}  |v_0-z|^{\gamma+2} f(z)^{1/2} \dd z  \\
&\leq K^{1/2} e^{-\mu|v_0|^2/8} |v_0|^2 \left(\int_{B_r(v_0)} |v_0-z|^{2(\gamma+2)} \dd z\right)^{1/2} \left(\int_{B_r(v_0)} f(z)\dd z\right)^{1/2}\\
&\lesssim K^{1/2} e^{-\mu |v_0|^2/8}|v_0|^2 r^{\gamma+7/2} \|f\|_{L^1}^{1/2}\\
&\lesssim K^{1/2} e^{-\mu|v_0|^2/8} |v_0|^{\gamma+11/2}\|f\|_{L^1}^{1/2}.
\end{split}
\]
Outside of $B_r(v_0)$, we use the energy bound, $\sin^2\theta_{v_0,z} \leq 1$, and $|v_0-z|\geq |v_0|/2$:
\[
|v_0|^2 \int_{\R^3\setminus B_r(v_0)} |z|^2 \sin^2\theta_{v_0,z} |v_0-z|^{\gamma} f(z) \dd z \lesssim  |v_0|^{\gamma+2} E_0 .
\]
For the last term on the right in \eqref{e:d2term}, we use the lower bounds of Lemma \ref{l:a-lower}. Overall, we have
\begin{equation}\label{e:traf}
\begin{split}
\tr(\bar a^f D_v^2 h)  &\lesssim  \mu h \left[ 2\mu  \left(\sqrt K e^{-\mu |v_0|^2/8} |v_0|^{\gamma+11/2} + E_0 |v_0|^{\gamma+2}\right)  - c_a  |v_0|^{\gamma+2}\right]\\
&\leq \mu h |v_0|^{\gamma+2} \left[ 2\mu \sqrt K e^{-\mu |v_0|^2/8} |v_0|^{7/2}+2\mu E_0 - c_a \right].
\end{split}
\end{equation}
Using the inequality $\sup_{x\geq 0} x^m e^{-\mu x^2} \lesssim \mu^{-m/2}$, followed by the general inequality \eqref{e:calc}, we have
\[
\begin{split}
2\mu \sqrt K e^{-\mu |v_0|^2/8} |v_0|^{7/2} &= 2\mu \sqrt K e^{-\mu |v_0|^2/16} |v_0|^{7/2} e^{-\mu|v_0|^2/16}\\
 &\lesssim \mu^{-3/4} \sqrt K  e^{-\mu |v_0|^2/16} \\
 & \lesssim \mu^{-7/4} e^{-1/(64\mu)} \sqrt K |v_0|^{-1}.
\end{split}
\]
From the definition \eqref{e:mu-def} of $\mu$, we obtain
\[
\mu \leq \frac 1 {33 \log(K/c_0)} < \frac 1 {65\log(\sqrt K/c_0)}.
\]
This implies $\mu^{5/4} e^{1/(64\mu)} \gtrsim e^{1/(65\mu)} \geq \sqrt K/c_0$, and if $|v_0|$ is large enough, we have
\[
|v_0| \geq \sqrt{\frac{\log 2}\mu} \gtrsim \frac{\sqrt K}{c_0 \mu^{7/4} e^{1/(64\mu)}},
\]
and therefore,
\[
2\mu \sqrt K e^{-\mu |v_0|^2/8} |v_0|^{7/2} \lesssim c_0 \leq \frac {c_a} 3,
\]
if $c_0$ in \eqref{e:mu-def} is chosen sufficiently small. Together with $\mu \leq c_a/(3E_0)$, this implies the right-hand side of \eqref{e:traf} is bounded by 
\[
\lesssim -c_a \mu h |v_0|^{\gamma+2},
\]
as desired.

Our assumption that $|v_0|\geq \sqrt{\log 2/\mu}$ is justified because otherwise, we would have, since $C_1 > 1$,
\[
h(v_0) = K e^{-\mu |v_0|^2} > \frac 1 2 K \geq \|f\|_{L^\infty} \geq f(t_0,x_0,v_0),
\]
 a contradiction.

Returning to \eqref{e:crossing-eqn}, we have shown
 \[
 0 \leq \left( - c_1 \mu |v_0|^{\gamma+2} + \bar c^f\right) h(v_0),
 \]
 for a constant $c_1>0$ proportional to $c_a$. To bound this right-hand side,  we consider the Coulomb ($\gamma = -3$) and non-Coulomb cases separately. 

In the Coulomb case, we have 
\begin{equation}\label{e:contra}
0 \leq \left(-c_1 \mu |v_0|^{-1} + f(t_0,x_0,v_0)\right)h(v_0) . 
\end{equation}
The following inequality for $\mu,s>0$ is easy to prove using calculus:
\begin{equation}\label{e:calc}
s e^{-\mu s^2} \leq \frac 1 {2\mu} e^{-1/(4\mu)}.
\end{equation}
Therefore, we have
\begin{equation*}
f(t_0,x_0,v_0) \leq K e^{-\mu |v_0|^2} \leq  \frac K {2\mu} e^{-1/(4\mu)} |v_0|^{-1}.
\end{equation*}
The function $\mu\mapsto \mu^2e^{1/(8\mu)}$ is uniformly bounded below by a positive constant $c$ on $(0,\infty)$, so if $c_0$ is chosen sufficiently small, our definition \eqref{e:mu-def} of $\mu$ implies
\begin{equation}\label{e:step2}
K \leq c_0 e^{1/(8\mu)} < c c_1 e^{1/(8\mu)} < c_1\mu^2 e^{1/(4\mu)},
\end{equation}
so that 
\begin{equation}\label{e:step3}
f(t_0,x_0,v_0) < c_1 \mu |v_0|^{-1},
\end{equation}
and the right-hand side of \eqref{e:contra} is negative, implying a contradiction. 

In the non-Coulomb case, instead of \eqref{e:contra}, we have
\begin{equation}\label{e:contra2}
0 \leq \left(-c_1 \mu |v_0 |^{\gamma+2} + \bar c^f(t_0,x_0,v_0)\right) h(v_0).
\end{equation}
To estimate the integral defining $\bar c^f$, we use \eqref{e:less-than-barrier} when $w$ is small, and the mass density bound when $w$ is large:
\[
\begin{split}
\bar c^f(t_0,x_0,v_0) &\lesssim \int_{B_{|v_0|/2}} |w|^\gamma K e^{-\mu|v_0-w|^2} \dd w  + \int_{\R^3\setminus B_{|v_0|/2}} |w|^\gamma f(t_0,x_0,v_0-w) \dd w\\
&\lesssim K e^{-\mu |v_0|^2/4} |v_0|^{\gamma+3} + M_0 |v_0|^\gamma.
\end{split}
\]
Using \eqref{e:calc}, we then have
\[
\begin{split}
\bar c^f(t_0,x_0,v_0) &\lesssim |v_0|^\gamma  \left( K e^{-\mu |v_0|^2/4} |v_0|^3 + M_0\right)\\
& \lesssim |v_0|^\gamma \left( \frac {2K} \mu e^{-1/\mu} |v_0|^2  +M_0 \right)\\
&\lesssim |v_0|^\gamma \left( \frac {c_1 \mu} 2 |v_0|^2 + M_0\right),
\end{split}
\]
where in the last line, we used a similar method to \eqref{e:step2} and \eqref{e:step3}. 
The expression inside parentheses is stricly less than $c_1 \mu |v_0|^2$ so long as 
\[
|v_0| > \sqrt{\frac{2M_0}{c_1 \mu}},
\]
which would imply the right side of \eqref{e:contra2} is negative, a contradiction.
On the other hand if $|v_0| \leq \sqrt{2M_0/(c_1\mu)}$, then we have
\[
h(v_0) = K e^{-\mu |v_0|^2} \geq K e^{-2M_0/c_1},
\]
and we choose $C_1= e^{2M_0/c_1}$ in the definition of $K$, so that this quantity is strictly greater than $\|f\|_{L^\infty}$, which means a crossing cannot occur in this case either.

We conclude $f< h$ on $[0,T]\times\R^6$, as desired. 
\end{proof}
%
%
%
%
%
%
%
%
%
%
%
%
%
%

Next, we use these  Gaussian bounds to improve the conditional dependence of our global upper bounds:

\begin{theorem}\label{t:inf}
Let $f:[0,T]\times\R^6\to [0,\infty)$ be a solution to the Landau equation \eqref{e:landau} satisfying the hypotheses of Proposition \ref{p:gaussian}. 
Then 
\[
\|f\|_{L^\infty([0,T]\times\R^6)}\leq C E_0^{-19\gamma/\delta},
\]
for a constant $C$ depending only on $\gamma$, $\delta$, the quantities in \eqref{e:cond}, and the constants $\ell$, $\rho$, $C_0$, and $\mu'$ corresponding to the initial data. 
\end{theorem}

Recall that $E_0$ is the bound for the energy density, as in the hypotheses of Proposition \ref{p:gaussian}.

\begin{proof}

As in the proof of Theorem \ref{t:other-upper}, we notice that $f$ is bounded on a small time interval $[0,T_*]$ by a constant depending only on the initial data, as a result of existence/uniqueness theory. 
Therefore, we may take a point $z_0 = (t_0,x_0,v_0)$ with $t_0 \geq T_*$, such that 
\[
f(t_0,x_0,v_0) \geq \frac 1 2 \|f\|_{L^\infty([0,T]\times \R^6)} = \frac 1 2 L_0.
\]
Applying Lemma \ref{l:Qr-est-scaling} and using $t_0 \geq T_*$, we have
\begin{equation}\label{e:decay-v0-eqn}
\begin{split}
f(t_0,x_0,v_0) &\lesssim \langle v_0 \rangle^{-19\gamma/2} L_0^{p/2} \|f\|_{L^\infty(Q_r(z_0))}^{(2-p-\delta)/2} \|f\|_{L^\infty_{t,x} L^{p+\delta}_v(Q_r(z_0))}^{(p+\delta)/2}\\
& \lesssim \langle v_0\rangle^{-19\gamma/2}L_0^{p/2} \|f\|_{L^\infty(Q_r(z_0))}^{(2-p-\delta)/2},
\end{split}
\end{equation}
after absorbing the norm $\|f\|_{L^\infty_{t,x} L^{p+\delta}_v}$ into the implied constant.

If $|v_0|\leq 2$, then 
since $f(t_0,x_0,v_0) \geq L_0/2$, this implies $L_0^{\delta/2}$ is bounded above by a constant depending only on $\delta$, the initial data, and the $L^\infty_{t,x} L^1_v$ and $L^\infty_{t.x} L^p_v$ norms of $f$. Therefore, we may assume $|v_0|>2$ for the remainder of the proof.

 In this case, to obtain an upper bound that is independent of $|v_0|$, we need to use the Gaussian decay of $f$. 
 Applying the Gaussian upper bound from Proposition \ref{p:gaussian} to \eqref{e:decay-v0-eqn}, we obtain, since $|v|^2 \geq |v_0|^2/4$ in $Q_r(z_0)$,
\[
f(t_0,x_0,v_0) 
\lesssim |v_0|^{-19\gamma/2} L_0^{p/2}\left( e^{-\mu |v_0|^2/4} \right)^{(2-p-\delta)/2} 
\lesssim |v_0|^{-19\gamma/2} L_0^{(2 - \delta)/2} e^{-(2-p-\delta)\mu|v_0|^2/8}.
\]
where we used that $K\lesssim L_0$, where $K$ is the constant from Proposition \ref{p:gaussian}. Using the inequality $x^m e^{-\mu x^2}\lesssim \mu^{-m/2}$, we then have
\[
f(t_0,x_0,v_0) \lesssim L_0^{(2-\delta)/2} \mu^{19\gamma/4}.
\]
Recalling the definition of $\mu$ in Proposition \ref{p:gaussian}, we may assume $\mu < \mu'/2$ since otherwise, $\mu$ is independent of $E_0$ and $L_0$, and the current proof is easier. We then have
\[
\mu = \min\left\{\frac {c_0}{E_0} , \frac 1 {33\log(K/c_0)}\right\} \gtrsim \frac 1 {E_0 \log(K)},
\]
since we can assume $E_0, K \gtrsim 1$. Similarly, we may assume $K \lesssim L_0$ since the other case, where $K$ is determined only by the initial data, is simpler. We then have
\[
\mu \gtrsim \frac 1 {E_0 \log(L_0)} \gtrsim \frac{L_0^{-\delta/(19\gamma)}} {E_0},  
\]
which yields, since $\gamma <0$, 
\[
f(t_0,x_0,v_0) \lesssim L_0^{(2-\delta)/2} E_0^{-19\gamma/4} L_0^{\delta/4} = L_0^{1-\delta/4} E_0^{-19\gamma/4}.
\]
By the choice of $(t_0,x_0,v_0)$, this implies
\[
L_0 \leq C E_0^{-19\gamma/\delta},
\]
for a constant $C>0$ depending only on $\delta$, the initial data, and the constants in \eqref{e:cond}. 
\end{proof}

\section{Bound for the energy density}\label{s:energy}

In this last section, we show that the upper bound on the energy density in the above estimates can be replaced by a bound on the $s$ moment for some small $s>0$. 

\begin{proposition}\label{p:E}
Let $f$ be a solution of the Landau equation on $[0,T]\times\R^6$ satisfying the hypotheses of Theorem \ref{t:upper}. Let 
\[
E_0 = \sup_{t,x} \int_{\R^3} |v|^2 f(t,x,v)\dd v,
\]
and let $s\in (0,2)$ be arbitrary. 

Then $E_0$ is bounded above by a constant depending only on $\gamma$, $\delta$, $s$, the initial data, the $L^\infty_{t,x} L^1_v$ and $L^\infty_{t,x}L^{p+\delta}_v$ norms of $f$, and
\[
\sup_{t,x} \int_{\R^3} |v|^s f(t,x,v) \dd v.
\]
\end{proposition}

\begin{proof}
Combining the $L^\infty$ bound of Theorem \ref{t:inf} with the Gaussian decay estimate of Proposition \ref{p:gaussian}, we obtain
\begin{equation}\label{e:good-gaussian}
f(t,x,v) \leq K e^{-\mu |v|^2}.
\end{equation}
As in the proof of Theorem \ref{t:inf}, we can assume $K\lesssim \|f\|_{L^\infty([0,T]\times\R^6)} \leq C E_0^{-19\gamma/\delta}$, since otherwise $K$ is independent of $E_0$, and the proof becomes simpler. Similarly, for $\mu$ we may assume
\[
\mu = \min\left\{\frac {c_0} {E_0}, \frac {1} {33\log(K/c_0)}\right\}  \gtrsim \frac {1} {E_0},
\]
since $\log(K/c_0) \approx \log(E_0) \lesssim E_0$. 
Here, the implied constants depend only on the quantities in the statement of the lemma.

Let $\theta =-19\gamma/\delta$. For any $s\in (0,2)$ and $q>1$, estimate \eqref{e:good-gaussian} and H\"older's inequality imply, with $q' = q/(q-1)$,
\[
\begin{split}
\int_{\R^3} |v|^2 f(t,x,v) \dd v &\leq K^{1/q} \int_{\R^3} |v|^{s/q'} |v|^{2-s/q'} f(t,x,v)^{q'} e^{-\mu|v|^2/q} \dd v\\
&\leq K^{1/q} \left(\int_{\R^3} |v|^{s} f(t,x,v) \dd v\right)^{1/q'} \left( \int_{\R^3} e^{-\mu|v|^2} |v|^{2q - s (q-1)} \dd v\right)^{1/q}\\
&\lesssim E_0^{\theta/q} \|f\|_{L^\infty_{t,x}(L^1_{s})_v}^{1/q'} \mu^{-1+s/2 - (s+3)/(2q) } C_{q,s}.
\end{split}
\]
Here, we use the notation $\|f\|_{L^\infty_{t,x}(L^1_s)_v} = \sup_{t,x}\int_{\R^3} |v|^s f(t,x,v)\dd v$, as well as 
\[
C_{q,s} = \left(\int_{\R^3} e^{-|w|^2} |w|^{2q - s(q-1)} \dd w\right)^{1/q}, 
\]
which depends only on $s$ and $q$. 
With 
$\mu \gtrsim E_0^{-1}$, 
we have
\[
\int_{\R^3} |v|^2 f(t,x,v) \dd v \lesssim C_{q,s} E_0^{\theta/q + 1 - s/2 +(s+3)/(2q)} \|f\|_{L^\infty_{t,x}(L^1_s)_v}^{1/q'},
\]
so we choose
\[
q = \frac 2 s \left( 2\theta + s + 3\right),
\]
so that $\theta/q+ (s+3)/(2q) = s/4$, and the exponent of $E_0$ becomes $1 - s/4$.   Taking the supremum over $t$ and $x$, we have
\[
E_0 \lesssim  E_0^{1 - s/4} \|f\|_{L^\infty_{t,x}(L^1_s)_v}^{1/q'},
\]
or 
\[
E_0 \lesssim \|f\|_{L^\infty_{t,x}(L^1_s)_v}^{4/(s q')}.
\]
\end{proof}

Combining Theorem \ref{t:inf} with Proposition \ref{p:E}, we obtain Theorem \ref{t:upper}.

\bibliographystyle{abbrv}
\bibliography{landau-rev}




\end{document}